\documentclass[12pt]{amsart}

\usepackage[psamsfonts]{amssymb}
\usepackage{mathrsfs}

\theoremstyle{plain}
\newtheorem{thm}[equation]{Theorem}

\newtheorem{defi}[equation]{\it Definition}

\newtheorem{rem}[equation]{\it Remark}

\newtheorem{lem}[equation]{Lemma}

\newtheorem*{Main-Theorem}{Theorem \ref{thm3}}

\address{Department of Mathematics, Graduate School of Science, Osaka University, 1-1 Machikaneyama, Toyonaka, Osaka 560-0043, JAPAN}
\email{h-tadano@cr.math.sci.osaka-u.ac.jp}
\title{Gap theorems for Ricci-harmonic solitons}
\dedicatory{}
\author{Homare TADANO}
\date{December 20, 2014}
\thanks{This work was supported by Moriyasu Graduate Student Scholarship Foundation}

\begin{document}

\begin{abstract}
In the present paper, by using estimates for the generalized Ricci curvature, we shall give some gap theorems for Ricci-harmonic solitons showing some necessary and sufficient conditions for the solitons to be harmonic-Einstein. Our results may be regarded as a generalization of recent works by H. Li, and M. Fern\'{a}ndez-L\'{o}pez and E. Garc\'{i}a-R\'{i}o.
\end{abstract}

\maketitle

\numberwithin{equation}{section}

\section{Introduction}

The geometric heat flow is one of the most powerful tool in mathematics, and has been studied extensively. Such a study originated from the celebrated work \cite{ES} of the harmonic map heat flow by Eells and Sampson.

\subsection{The harmonic map heat flow}

Let $(\mathcal{M}, g)$ and $(\mathcal{N}, h)$ be two complete Riemannian manifolds of dimension $m$ and $n$, respectively. We assume that $\mathcal{M}$ is compact. In 1964, Eells and Sampson \cite{ES} introduced the \textit{harmonic map heat flow}
\begin{equation}\label{HMF}
\frac{\partial \phi}{\partial t}(t) = \tau_{g} \phi(t), \quad \phi(0) = \phi_{0}
\end{equation}
to find a harmonic map homotopic to a given map $\phi_{0} : (\mathcal{M}, g) \rightarrow (\mathcal{N}, h)$, where $\phi(t) : (\mathcal{M}, g) \rightarrow (\mathcal{N}, h)$ is a family of smooth maps between two manifolds and $\tau_{g} \phi(t) := \mathrm{trace} \nabla d \phi(t)$ denotes the tension field of $\phi(t)$ with respect to the metric $g$. They solved the problem when the target manifold $(\mathcal{N}, h)$ has non-positive sectional curvature. We refer the reader to the survey articles \cite{EL1, EL2} for basic facts about the harmonic map.

\medskip

Inspired by the idea of Eells and Sampson \cite{ES}, Hamilton \cite{Hamilton} introduced the Ricci flow.

\subsection{The Ricci flow}

Let $(\mathcal{M}, g(t))$ be a family of compact $m$-dimensional Riemannian manifolds with Riemannian metrics $g(t)$ evolving by the following \textit{Ricci flow}:
\begin{equation}\label{RF}
\frac{\partial g}{\partial t}(t) = -2 \operatorname{Ric}_{g(t)}, \quad g(0) = g_{0}, 
\end{equation}
where $\operatorname{Ric}_{g(t)}$ denotes the Ricci tensor with respect to the metric $g(t)$ and $g_{0}$ is an initial Riemannian metric on $\mathcal{M}$. In 2002, Perelman \cite{P} proved the Poincar\'{e} conjecture by using the Ricci flow. Recently, the Ricci flow has achieved great success to find canonical metrics on Riemannian manifolds, and has been an important tool in differential geometry. We refer the reader to the book \cite{CLN} for basic facts about the Ricci flow.

\medskip

The above two flows were combined by M\"{u}ller in the following way:

\subsection{The Ricci-harmonic flow}

Let $(\mathcal{M}, g(t))$ be a family of complete $m$-dimensional Riemannian manifolds with Riemannian metrics $g(t)$ evolving by the following \textit{Ricci-harmonic flow}:
\begin{equation}\label{RHF}
\left\{
\begin{aligned}
\frac{\partial g}{\partial t}(t) & = - 2 \operatorname{Ric}_{g(t)} + 2 \alpha(t) \nabla \phi(t) \otimes \nabla \phi(t), \quad g(0) = g_{0}, \\
\frac{\partial \phi}{\partial t}(t) & = \tau_{g(t)} \phi(t), \quad \phi(0) = \phi_{0}, 
\end{aligned}
\right.
\end{equation}
where $\alpha(t) \geqslant 0$ is a non-negative time-dependent coupling constant, $\phi(t) : (\mathcal{M}, g(t)) \rightarrow (\mathcal{N}, h)$ is a family of smooth maps between $(\mathcal{M}, g(t))$ and a fixed complete Riemannian manifold $(\mathcal{N}, h)$, and $\nabla \phi(t) \otimes \nabla \phi(t) := \phi(t)^{*} h$ is the pull-back of the metric $h$ via $\phi(t)$. The flow was introduced by M\"{u}ller \cite{Muller2} to find a harmonic map between two manifolds. The short time existence of the flow is proved. A typical example would be the List's extended Ricci flow \cite{List} introduced by its connection to general relativity that the stationary points correspond to the static Einstein vacuum equations. More examples can be found in \cite{Muller2}.

\medskip

Surprisingly, the Ricci-harmonic flow has better properties than the harmonic map heat flow or the Ricci flow alone. In fact, if the curvature of $(\mathcal{M}, g(t))$ is uniformly bounded, then by taking the coupling constant $\alpha(t)$ suitably, we can control the derivative of the map $\phi(t)$, and the long-time existence for the flow follows. Hence, the Ricci-harmonic flow may be useful to find a harmonic map between two Riemannian manifolds. Furthermore, many authors pointed out that the Ricci-harmonic flow shares good properties with the Ricci flow. For instance, as in the Perelman's work \cite{P}, M\"{u}ller introduced the $\mathcal{F}_{\alpha}$-functional and the $\mathcal{W}_{\alpha}$-functional for the Ricci-harmonic flow and established corresponding theory, such as, the no breather theorem, the non-collapsing theorem \cite{Muller2}, the monotonicity formula \cite{Muller1}. By denoting $\mathscr{S}(t) := \operatorname{Ric}_{g(t)} - \alpha(t) \nabla \phi(t) \otimes \nabla \phi(t)$, the first equation in (\ref{RHF}) can be written as
\[
\frac{\partial g}{\partial t}(t) = - 2 \mathscr{S}(t).
\]
As with the Ricci flow, under (\ref{RHF}), some differential Harnack inequalities for the heat equation are obtained \cite{Guo-He, Zhu}.

\medskip

In the present paper, we focus on the soliton-type solution of the flow (\ref{RHF}).

\medskip

\subsection{The Ricci-harmonic soliton}

Let $(\mathcal{M}, g)$ and $(\mathcal{N}, h)$ be two static complete Riemannian manifolds of dimension $m$ and $n$, respectively. Let $\phi : (\mathcal{M}, g) \rightarrow (\mathcal{N}, h)$ be a smooth map between $(\mathcal{M}, g)$ and $(\mathcal{N}, h)$, $f : \mathcal{M} \rightarrow \mathbb{R}$ a smooth function on $\mathcal{M}$, and $\lambda \in \mathbb{R}$ a real number.

\begin{defi}[Williams \cite{Williams}]\rm
The $4$-tuple $((\mathcal{M}, g), (\mathcal{N}, h), \phi, \lambda)$ is called \textit{harmonic-Einstein} if it satisfies the following coupled system:
\begin{equation}\label{HE}
\left\{
\begin{aligned}
& \operatorname{Ric}_{g} - \alpha \nabla \phi \otimes \nabla \phi = \lambda g, \\
& \tau_{g} \phi = 0, 
\end{aligned}
\right.
\end{equation}
where $\operatorname{Ric}_{g}$ denotes the Ricci tensor with respect to the metric $g$, $\alpha \in \mathbb{R}$ is a constant, $\nabla \phi \otimes \nabla \phi := \phi^{*}h$ is the pull-back of the metric $h$ via $\phi$, and $\tau_{g} \phi := \mathrm{trace} \nabla d \phi$ denotes the tension field of $\phi$ with respect to $g$.
\end{defi}

\begin{defi}[M\"{u}ller \cite{Muller2}]\label{rem}\rm
The $5$-tuple $((\mathcal{M}, g), (\mathcal{N}, h), \phi, f, \lambda)$ is called a \textit{Ricci-harmonic soliton} if it satisfies the following coupled system:
\begin{equation}\label{RHS}
\left\{
\begin{aligned}
& \operatorname{Ric}_{g} - \alpha \nabla \phi \otimes \nabla \phi + \operatorname{Hess} f = \lambda g, \\
& \tau_{g} \phi = \left< \nabla \phi, \nabla f \right>, 
\end{aligned}
\right.
\end{equation}
where $\alpha \geqslant 0$ is a non-negative constant and $\operatorname{Hess} f$ denotes the Hessian of $f$. We say that the soliton $((\mathcal{M}, g), (\mathcal{N}, h), \phi, f, \lambda)$ is \textit{shrinking}, \textit{steady} and \textit{expanding} described as $\lambda > 0, \lambda = 0$, and $\lambda < 0$, respectively. If $f$ is constant in (\ref{RHS}), then the soliton is harmonic-Einstein. In such a case, we say that the soliton is \textit{trivial}.
\end{defi}

\begin{defi}\rm
A Ricci-harmonic soliton $((\mathcal{M}, g), (\mathcal{N}, h), \phi, f, \lambda)$ is \textit{normalized} if its potential function $f$ satisfies
\begin{equation}\label{normalized}
\int_{\mathcal{M}} f = 0.
\end{equation}
\end{defi}

The soliton (\ref{RHS}) above corresponds to a self-similar solution for the coupled system (\ref{RHF}). Note that, if $(\mathcal{N}, h) = (\mathbb{R}, dr^{2})$ and $\phi : (\mathcal{M}, g) \rightarrow (\mathbb{R}, dr^{2})$ is a constant function in (\ref{RHS}), then the soliton is exactly a gradient Ricci soliton. Recently, corresponding theory has been established for Ricci-harmonic solitons. For instance, Williams \cite{Williams} proved that any steady or expanding Ricci-harmonic soliton with a compact domain manifold must be trivial. Following the work by Cao and Zhou \cite{Cao-Zhou}, Yang and Shen \cite{Yang-Shen} gave a volume growth estimate for complete non-compact domain manifolds of shrinking Ricci-harmonic solitons. Recently, the author obtained a lower diameter bound for compact domain manifolds of shrinking Ricci-harmonic solitons \cite{Tadano2}.

\medskip

The aim of the present paper is to give some gap theorems for Ricci-harmonic solitons by showing necessary and sufficient conditions for the solitons to be harmonic-Einstein. The same observations were made for the Ricci soliton on K\"{a}hler manifolds \cite{Li} and Riemannian manifolds \cite{FL-GR}. Our main theorem is the following:

\begin{Main-Theorem}
Let $((\mathcal{M}, g), (\mathcal{N}, h), \phi, f, \lambda)$ be a shrinking Ricci-harmonic soliton satisfying {\rm (\ref{RHS})}. Suppose that the domain manifold $\mathcal{M}$ is compact and the soliton is normalized in sense of {\rm (\ref{normalized})}. Then there exists a non-negative constant $\delta \ll 1$ such that, if
\[
\operatorname{Ric}_{g} - \alpha \nabla \phi \otimes \nabla \phi \geqslant (\lambda - \delta) g, 
\]
then the soliton is harmonic-Einstein.
\end{Main-Theorem}

\begin{rem}\rm
In Theorem \ref{thm3} above, $\delta \geqslant 0$ depends only on the dimension $m$ of $\mathcal{M}$ and the average $\Lambda = \frac{1}{\mathrm{vol}(\mathcal{M}, g)} \int_{\mathcal{M}} \| \nabla f \|^{2}$ of the $L^{2}$-norm of $\nabla f$. Moreover, the proof shows that $\delta$ can be expressed explicitly in terms of $m$ and $\Lambda$.
\end{rem}

Roughly speaking, the above result shows that, if the generalized Ricci curvature of a Ricci-harmonic soliton is sufficiently close to that of a trivial soliton, then the soliton must be trivial. Hence, Theorem \ref{thm3} above gives us a gap phenomenon between non-trivial Ricci-harmonic solitons and trivial ones. We conclude this introduction by remarking that, if $(\mathcal{N}, h) = (\mathbb{R}, dr^{2})$ and $\phi : (\mathcal{M}, g) \rightarrow (\mathbb{R}, dr^{2})$ is a constant function in (\ref{RHS}), then our Theorem \ref{thm3} gives us a new gap theorem for the Ricci soliton.

\medskip

This paper is organized as follows: In Section 2, by introducing notations and definitions, we give some lemmas playing an important role in this paper. Ending with Section 3, a proof for Theorem \ref{thm3} will be given.

\section{Preriminaries}

In this section, we provide some formulas and lemmas playing an important role in the present paper. First, by taking the trace of the first equation in (\ref{RHS}), we have
\begin{equation}\label{fund-eq1-pre}
R - \alpha | \nabla \phi |^{2} + \Delta f = n \lambda, 
\end{equation}
where $R := g^{ij} R_{ij}$ denotes the scalar curvature on $(M, g)$.

\begin{lem}[M\"{u}ller \cite{Muller2}]\label{lem1}
Let $((\mathcal{M}, g), (\mathcal{N}, h), \phi, f, \lambda)$ be a shrinking Ricci-harmonic soliton satisfying {\rm (\ref{RHS})}. Then there exists some real constant $C_{1} \in \mathbb{R}$ such that
\begin{equation}\label{fund-eq2-pre}
R - \alpha | \nabla \phi |^{2} + | \nabla f |^{2} - 2 \lambda f = C_{1}.
\end{equation}
\end{lem}

\begin{proof}
First, taking a covariant derivative of the first equation in (\ref{RHS}), we obtain
\[
\nabla_{k} R_{ij} - \alpha \nabla_{k} (\nabla_{i} \phi \nabla_{j} \phi) + \nabla_{k} \nabla_{i} \nabla_{j} f = 0.
\]
Subtracting the same equation with indices $i$ and $k$ interchanged, we have
\[
\nabla_{k} R_{ij} - \nabla_{i} R_{kj} - \alpha (\nabla_{i} \phi \nabla_{k} \nabla_{j} \phi - \nabla_{k} \phi \nabla_{i} \nabla_{j} \phi) + R_{kijp} \nabla_{p} f = 0.
\]
Tracing just above with $g^{kj}$ yields
\[
\nabla_{j} R_{ij} - \nabla_{i} R - \alpha \left( \nabla_{i} \phi \tau_{g} \phi - \frac{1}{2} \nabla_{i} | \nabla \phi |^{2} \right) + R_{ip} \nabla_{p} f = 0.
\]
Using the contracted second Bianchi identity $\nabla_{j} R_{ij} = \frac{1}{2} \nabla_{i} R$ and plugging in both equations of (\ref{RHS}) for $R_{ip}$ and for $\tau_{g} \phi$, we obtain
\[
- \frac{1}{2} \nabla_{i} (R - \alpha | \nabla \phi |^{2} + | \nabla f |^{2} - 2 \lambda f) = 0.
\]
This proves Lemma \ref{lem1}.
\end{proof}

For simplicity, we put
\begin{equation}\label{calS}
\mathscr{S} := \mathrm{Ric}_{g} - \nabla \phi \otimes \nabla \phi
\end{equation}
and refer to (\ref{calS}) as a \textit{generalized Ricci curvature}. We denote by
\begin{equation}\label{S}
S := \sum_{i = 1}^{m} \mathscr{S}(e_{i}, e_{i}) = R - \alpha | \nabla \phi |^{2}
\end{equation}
the trace of the generalized Ricci curvature, where $\{ e_{i} \}_{i = 1}^{m}$ is an orthonormal frame of $(\mathcal{M}, g)$. We call (\ref{S}) a \textit{generalized scalar curvature}. In view of (\ref{S}), the equation (\ref{fund-eq1-pre}) and (\ref{fund-eq2-pre}) are written as
\begin{equation}\label{fund-eq1}
S + \Delta f = m \lambda
\end{equation}
and
\begin{equation}\label{fund-eq2}
S + \| \nabla f \|^{2} - 2 \lambda f = C_{1}, 
\end{equation}
respectively. Then by taking the difference of (\ref{fund-eq1}) and (\ref{fund-eq2}), we have
\begin{equation}\label{fund-eq3}
\Delta f - \| \nabla f \|^{2} + 2 \lambda f = C_{2}, 
\end{equation}
where $C_{2} = m \lambda - C_{1}$. By (\ref{fund-eq2}) and the first equation in (\ref{RHS}), we obtain
\begin{equation}\label{fund-eq4}
\mathscr{S}(\nabla f, \cdot) = \frac{1}{2} dS.
\end{equation}
The following lemma plays a crucial role in this paper:

\begin{lem}\label{lem2}
Let $((\mathcal{M}, g), (\mathcal{N}, h), \phi, f, \lambda)$ be a shrinking Ricci-harmonic soliton satisfying {\rm (\ref{RHS})}. Suppose that the domain manifold $\mathcal{M}$ is compact. Then the following holds:
\begin{equation}\label{lem2-eq1}
\int_{\mathcal{M}} \| \mathrm{Hess} f \|^{2}\leqslant \int_{\mathcal{M}} \mathscr{S}(\nabla f, \nabla f) = \frac{1}{2} \int_{\mathcal{M}} (\Delta f)^{2}, 
\end{equation}
\begin{equation}\label{lem2-eq2}
\int_{\mathcal{M}} (\Delta f)^{2} = \int_{\mathcal{M}} ((m + 2) \lambda - S) \| \nabla f \|^{2}, \mbox{ and} \\
\end{equation}
\begin{equation}\label{lem2-eq3}
\| \nabla f \|^{2} \leqslant S_{\mathrm{max}} - S, 
\end{equation}
where $S_{\mathrm{max}}$ denotes the maximum value of the generalized scalar curvature.
\end{lem}

\begin{proof}
First, we have
\[
\begin{aligned}
\frac{1}{2} \Delta \| \nabla f \|^{2} & = \| \mathrm{Hess} f \|^{2} + g(\nabla f, \nabla \Delta f) + \mathrm{Ric}_{g}(\nabla f, \nabla f) \\
& = \| \mathrm{Hess} f \|^{2} - g(\nabla f, \nabla S) + \operatorname{Ric}_{g}(\nabla f, \nabla f) & & \mbox{(due to (\ref{fund-eq1}))} \\
& \geqslant \| \mathrm{Hess} f \|^{2} - \mathscr{S}(\nabla f, \nabla f), & & \mbox{(due to (\ref{fund-eq4}))}
\end{aligned}
\]
where the last inequality follows from $\alpha \nabla \phi \otimes \nabla \phi(\nabla f, \nabla f) \geqslant 0$. By integrating both sides of the last inequality, we have 
\[
\int_{\mathcal{M}} \| \mathrm{Hess} f \|^{2} \leqslant \int_{\mathcal{M}} \mathscr{S}(\nabla f, \nabla f).
\]
Then
\[
\begin{aligned}
\int_{\mathcal{M}} (\Delta f)^{2} & = - \int_{\mathcal{M}} S \Delta f & & \mbox{(due to (\ref{fund-eq1}))} \\
& = \int_{\mathcal{M}} g(\nabla S, \nabla f) = 2 \int_{\mathcal{M}} \mathscr{S}(\nabla f, \nabla f), & & \mbox{(due to (\ref{fund-eq4}))}
\end{aligned}
\]
which proves (\ref{lem2-eq1}). Secondly, we have
\[
\begin{aligned}
\int_{\mathcal{M}} (\Delta f)^{2} & = 2 \int_{\mathcal{M}} \mathscr{S}(\nabla f, \nabla f) & & \mbox{(due to (\ref{lem2-eq1}))} \\
& = 2 \lambda \int_{\mathcal{M}} \| \nabla f \|^{2} - 2 \int_{\mathcal{M}} \mathrm{Hess} f(\nabla f, \nabla f) \\
& = 2 \lambda \int_{\mathcal{M}} \| \nabla f \|^{2} - \int_{\mathcal{M}} g(\nabla f, \nabla \| \nabla f \|^{2}) \\
& = 2 \lambda \int_{\mathcal{M}} \| \nabla f \|^{2} + \int_{\mathcal{M}} \| \nabla f \|^{2} \Delta f \\
& = (m + 2) \lambda \int_{\mathcal{M}} \| \nabla f \|^{2} - \int_{\mathcal{M}} S \| \nabla f \|^{2} & & \mbox{(due to (\ref{fund-eq1}))}
\end{aligned}
\]
and (\ref{lem2-eq2}) follows. Here, in the last third equality above, we have used a property of the Levi-Civita connection. Finally, in order to prove (\ref{lem2-eq3}), recall from (\ref{fund-eq2}) that $2 \lambda f = S + \| \nabla f \|^{2} - C_{1}$ for some constant $C_{1}$. By compactness of the domain manifold $\mathcal{M}$, there exists some global maximum point $p \in \mathcal{M}$ of the potential function. Then, it follows from (\ref{fund-eq2}) that for any point $x \in \mathcal{M}$, 
\[
2 \lambda f(p) = S(p) - C_{1} \geqslant 2 \lambda f(x) = S(x) + \| \nabla f \|^{2} (x) - C_{1}, 
\]
and hence, $S(p) \geqslant S(x)$. Therefore, the generalized scalar curvature also attains its maximum at $p$, and we obtain (\ref{lem2-eq3}).
\end{proof}

\section{Gap Theorems}

In this section, we extend gap theorems for compact K\"{a}hler-Ricci solitons \cite{Li} and for compact gradient Ricci solitons \cite{FL-GR} to the case of Ricci-harmonic solitons with a compact domain manifold. First, note that, if a Ricci-harmonic soliton $((\mathcal{M}, g), (\mathcal{N}, h), \phi, \lambda)$ is trivial, then the generalized scalar curvature $S$ on $\mathcal{M}$ satisfies
\[
S = m \lambda, \quad \mbox{therefore}, \quad S_{\mathrm{max}} = m \lambda.
\]
The following result characterizes triviality for a Ricci-harmonic soliton by using an upper bound for $S_{\mathrm{max}} - m \lambda$ in terms of the average of the $L^{2}$-norm of $\nabla f$:

\begin{thm}\label{thm1}
Let $((\mathcal{M}, g), (\mathcal{N}, h), \phi, f, \lambda)$ be a shrinking Ricci-harmonic soliton satisfying {\rm (\ref{RHS})}. Suppose that the domain manifold $\mathcal{M}$ is compact. Then the soliton is harmonic-Einstein if and only if
\[
S_{\mathrm{max}} - m \lambda \leqslant \left( 1 + \frac{2}{m} \right) \frac{1}{\mathrm{vol}(M, g)} \int_{\mathcal{M}} \| \nabla f \|^{2}.
\]
\end{thm}

\begin{proof}
The result is obvious if the soliton is trivial, since in such a case the potential function is constant. Conversely, we have
\[
\begin{aligned}
\int_{ \mathcal{M}} (\Delta f)^{2} & = (m + 2) \lambda \int_{\mathcal{M}} \| \nabla f \|^{2} - \int_{\mathcal{M}} S \| \nabla f \|^{2} & & \mbox{(due to (\ref{lem2-eq2}))} \\
& \geqslant (m + 2) \lambda \int_{\mathcal{M}} \| \nabla f \|^{2} - S_{\mathrm{max}} \int_{\mathcal{M}} S + \int_{\mathcal{M}} S^{2} & & \mbox{(due to (\ref{lem2-eq3}))} \\
& = (m + 2) \lambda \int_{\mathcal{M}} \| \nabla f \|^{2} - m \lambda S_{\mathrm{max}} \mathrm{vol}(\mathcal{M}, g) & & \mbox{(due to (\ref{fund-eq1}))} \\
& \quad + m^{2} \lambda^{2} \hspace{0.25mm} \mathrm{vol}(\mathcal{M}, g) + \int_{\mathcal{M}} (\Delta f)^{2}, 
\end{aligned}
\]
which yields
\[
S_{\mathrm{max}} - m \lambda \geqslant \left( 1 + \frac{2}{m} \right) \frac{1}{\mathrm{vol}(\mathcal{M}, g)} \int_{\mathcal{M}} \| \nabla f \|^{2}.
\]
Hence, by the assumption in Theorem \ref{thm1}, the equality just above must be achieved. This shows that the equality in (\ref{lem2-eq3}) must also attain. Therefore, by (\ref{fund-eq2}) we have
\[
2 \lambda f - S + C_{1} = \| \nabla f \|^{2} = S_{\mathrm{max}} - S, 
\]
equivalently, $2 \lambda f = S_{\mathrm{max}} - C_{1}$. Hence, the potential function is constant and the soliton is trivial.
\end{proof}

In view of (\ref{RHS}) and (\ref{lem2-eq1}), on any compact Ricci-harmonic soliton $((\mathcal{M}, g), (\mathcal{N}, h), \phi, f, \lambda)$, we have
\[
\int_{\mathcal{M}} \| \mathscr{S} - \lambda g \|^{2} \leqslant \int_{\mathcal{M}} \mathscr{S}(\nabla f, \nabla f).
\]
Hence, if $\int_{\mathcal{M}} \mathscr{S}(\nabla f, \nabla f) \leqslant 0$, then the soliton must be trivial. Therefore, the quantity $\int_{\mathcal{M}} \mathscr{S}(\nabla f, \nabla f)$ measures the difference of the soliton from being harmonic-Einstein. The following result characterizes triviality of the soliton by giving an upper bound of this quantity:

\begin{thm}\label{thm2}
Let $((\mathcal{M}, g), (\mathcal{N}, h), \phi, f, \lambda)$ be a shrinking Ricci-harmonic soliton satisfying {\rm (\ref{RHS})}. Suppose that the domain manifold $\mathcal{M}$ is compact. Then the soliton is harmonic-Einstein if and only if
\[
\int_{\mathcal{M}} \mathscr{S}(\nabla f, \nabla f) \leqslant \frac{\lambda_{1}}{2} \int_{\mathcal{M}} \| \nabla f \|^{2}, 
\]
where $\lambda_{1}$ denotes the first non-zero eigenvalue of the Laplacian for $(\mathcal{M}, g)$.
\end{thm}

\begin{proof}
The result is obvious if the soliton is trivial. To prove that the soliton is trivial, we first note that any smooth function $f \in \mathcal{C}^{\infty}(\mathcal{M})$ satisfies
\[
\lambda_{1} \int_{\mathcal{M}} \| \nabla f \|^{2} \leqslant \int_{\mathcal{M}} (\Delta f)^{2}.
\]
Therefore, by (\ref{lem2-eq1}) and the assumption in Theorem \ref{thm2}, we see that the potential function $f$ satisfies
\begin{equation}
\lambda_{1} \int_{\mathcal{M}} \| \nabla f \|^{2} = \int_{\mathcal{M}} (\Delta f)^{2}, 
\end{equation}
and hence, the function $f$ is an eigenfunction of the Laplacian for $(\mathcal{M}, g)$ associated with $\lambda_{1}$. Then, it follows from (\ref{fund-eq4}) that $(2 \lambda - \lambda_{1}) f = \| \nabla f \|^{2} + C_{2}$. In the case that $2 \lambda - \lambda_{1} \neq 0$, since $\nabla f$ vanishes at any local extrema of $f$, we have
\[
f_{\mathrm{max}} = f_{\mathrm{min}} = \frac{C_{2}}{2 \lambda - \lambda_{1}}, 
\]
which shows that the potential function is constant and the soliton is trivial. In the case that $2 \lambda - \lambda_{1} = 0$, since we have $0 = \| \nabla f \|^{2} + C_{2}$, the same argument as in the previous case allows us to obtain $C_{2} = 0$, which shows that the potential function is also constant and the soliton is trivial.
\end{proof}

\begin{rem}\rm
In \cite[Theorem 1.1]{Li} and \cite[Theorem 2.6]{FL-GR}, symbolic gap theorems for Ricci solitons are given. However, the same argument does not hold for Ricci-harmonic solitons unless $\phi$ is a constant map.
\end{rem}

The following result shows that, if the generalized Ricci curvature of a Ricci-harmonic soliton is sufficiently close to that of a trivial soliton, then the soliton must be trivial. See \cite{Li, FL-GR, Tadano} for the same gap theorems for compact K\"{a}hler-Ricci solitons, compact Riemannian Ricci solitons, and compact Sasaki-Ricci solitons, respectively.

\begin{thm}\label{thm3}
Let $((\mathcal{M}, g), (\mathcal{N}, h), \phi, f, \lambda)$ be a shrinking Ricci-harmonic soliton satisfying {\rm (\ref{RHS})}. Suppose that the domain manifold $\mathcal{M}$ is compact and the soliton is normalized in sense of {\rm (\ref{normalized})}. Then, there exists a non-negative constant $\delta \ll 1$ such that if
\[
\operatorname{Ric}_{g} - \alpha \nabla \phi \otimes \nabla \phi \geqslant (\lambda - \delta) g, 
\]
then the soliton is harmonic-Einstein.
\end{thm}

\begin{proof}
We here assume that the soliton is non-trivial and deduce a contradiction. For simplicity, we put
\[
\Lambda := \frac{1}{\mathrm{vol}(\mathcal{M}, g)} \int_{\mathcal{M}} \| \nabla f \|^{2}.
\]
Then by (\ref{fund-eq4}) and the normalization (\ref{normalized}), we have
\begin{equation}\label{fund-eq5}
\Delta f - \| \nabla f \|^{2} + 2 \lambda f = - \Lambda.
\end{equation}
We assume that $\mathscr{S} \geqslant K g$ for some positive constant $K > 0$. Then, the generalized scalar curvature satisfies $S \geqslant mK$ and it follows from Myers theorem that
\begin{equation}\label{Myers}
\mathrm{diam}(\mathcal{M}, g) \leqslant \pi \sqrt{\frac{m - 1}{K}}.
\end{equation}

\begin{lem}\label{Lem}
The generalized scalar curvature is uniformly bounded from above, i.e., 
\begin{equation}\label{Lambda}
S < L(m, K, \Lambda), 
\end{equation}
where $L = L(m, K, \Lambda)$ is a constant depending only on the numbers $m, K$ and $\Lambda$. Moreover, 
\[
\lim_{K \rightarrow \lambda - 0} L(m, K, \Lambda) < + \infty.
\]
\end{lem}

\begin{proof}[Proof of Lemma {\rm \ref{Lem}}]
We first observe that
\[
\begin{aligned}
\| \nabla f \|^{2} & = \Delta f + 2 \lambda f + \Lambda & & \mbox{(due to (\ref{fund-eq5}))} \\
& = m \lambda - S + 2 \lambda f + \Lambda & & \mbox{(due to (\ref{fund-eq1}))} \\
& \leqslant m \lambda - mK + 2 \lambda f + \Lambda.
\end{aligned}
\]
Put $B := m \lambda - mK + \Lambda + 1$. Note that $B$ is constant. Then, it follows from the last inequality just above that $2 \lambda f + B \geqslant 1$, and hence, 
\begin{equation}\label{est}
\frac{\| \nabla f \|^{2}}{2 \lambda f + B} \leqslant 1 - \frac{1}{2 \lambda f + B} < 1.
\end{equation}
By compactness of the domain manifold $\mathcal{M}$, there exists some global minimum point $q \in \mathcal{M}$ of the potential function. By the normalized condition (\ref{normalized}), we see that $f(q) \leqslant 0$. Therefore, for any point $x \in \mathcal{M}$, 
\[
\begin{aligned}
\sqrt{2 \lambda f + B}(x) - \sqrt{2 \lambda f + B}(q) & \leqslant \left( \max_{\mathcal{M}} \left \| \nabla \sqrt{2 \lambda f + B} \right \| \right) \cdot \operatorname{diam}(\mathcal{M}, g) \\
& = \left( \max_{\mathcal{M}} \frac{\| \nabla f \|}{\sqrt{2 \lambda f + B}} \right) \cdot \lambda \cdot \mathrm{diam} (\mathcal{M}, g) \\
& < \lambda \cdot \pi \sqrt{\frac{m - 1}{K}}. \hspace{20mm} \mbox{(due to (\ref{est}) and (\ref{Myers}))}
\end{aligned}
\]
For simplicity, we put $c_{m} := \lambda \pi \sqrt{m - 1}$. Then, it follows from the above that
\begin{equation}\label{est2}
2 \lambda f(x) < \frac{(c_{m})^{2}}{K} + 2c_{m} \sqrt{\frac{B}{K}}.
\end{equation}
Therefore, 
\[
\begin{aligned}
S & = m \lambda - \Delta f & & \mbox{(due to (\ref{fund-eq1}))} \\
& = m \lambda + 2 \lambda f - \| \nabla f \|^{2} + \Lambda & & \mbox{(due to (\ref{fund-eq5}))} \\
& < m \lambda + \frac{(c_{m})^{2}}{K} + 2c_{m} \sqrt{\frac{B}{K}} + \Lambda. & & \mbox{(due to (\ref{est2}))}
\end{aligned}
\]
Hence, we can take the last number just above as $L(m, K, \Lambda)$.
\end{proof}

Now, we can finish the proof of Theorem \ref{thm3}. For simplicity, we put
\[
\Omega_{+} := \{ x \in \mathcal{M} : S(x) > m \lambda \} \quad \mbox{and} \quad \Omega_{-} := \{ x \in \mathcal{M} : S(x) < m \lambda \}, 
\]
respectively. Then, we have
\[
\begin{aligned}
\frac{1}{\mathrm{vol}(\mathcal{M}, g)} \int_{\mathcal{M}} (\Delta f)^{2} & = \frac{1}{\mathrm{vol}(\mathcal{M}, g)} \int_{\mathcal{M}} (S - m \lambda)^{2} \quad \mbox{(due to (\ref{fund-eq1}))} \\
& = \frac{1}{\mathrm{vol}(\mathcal{M}, g)} \int_{\Omega_{+}} (S - m \lambda)^{2} + \frac{1}{\mathrm{vol}(\mathcal{M}, g)} \int_{\Omega_{-}} (m \lambda - S)^{2} \\
& < \frac{1}{\mathrm{vol}(\mathcal{M}, g)} (L - m \lambda) \int_{\Omega_{+}} (S - m \lambda) + m^{2} \cdot (\lambda - K)^{2}.
\end{aligned}
\]
On the other hand, by integrating both sides of (\ref{fund-eq1}), we have
\[
0 = \int_{\Omega_{+}} (S - m \lambda) + \int_{\Omega_{-}} (S - m \lambda).
\]
Therefore, 
\[
\begin{aligned}
\frac{1}{\mathrm{vol}(\mathcal{M}, g)} \int_{\mathcal{M}} (\Delta f)^{2} & < \frac{1}{\mathrm{vol}(\mathcal{M}, g)} (L - m \lambda) \int_{\Omega_{-}} (m \lambda - S) + m^{2} \cdot (\lambda - K)^{2} \\
& \leqslant (L - m \lambda) \cdot m \cdot (\lambda - K) + m^{2} \cdot (\lambda - K)^{2}, 
\end{aligned}
\]
and hence, 
\begin{equation}\label{contradict}
\frac{1}{\mathrm{vol}(\mathcal{M}, g)} \int_{\mathcal{M}} (\Delta f)^{2} \rightarrow 0 \quad \mbox{as} \quad K \rightarrow \lambda - 0.
\end{equation}
However, we have
\[
\begin{aligned}
\frac{1}{\mathrm{vol}(\mathcal{M}, g)} \int_{\mathcal{M}} (\Delta f)^{2} & = \frac{2}{\mathrm{vol}(\mathcal{M}, g)} \int_{\mathcal{M}} \mathscr{S}(\nabla f, \nabla f) & & \mbox{(due to (\ref{lem2-eq1}))} \\
& \geqslant 2K \Lambda > 0, 
\end{aligned}
\]
which contradicts (\ref{contradict}) when $K$ is sufficiently close to $\lambda$. This proves Theorem \ref{thm3}.
\end{proof}

\end{document}